\documentclass[11pt]{article}
\usepackage{amsmath, amsthm, amssymb}
\usepackage{amscd}
\usepackage{amsfonts}
\usepackage{latexsym}
\usepackage{graphics}
\usepackage{epsf}
\usepackage{epsfig}
\usepackage{vmargin}
\usepackage{tikz}

\usetikzlibrary{arrows,automata}

\newcommand{\R}{\mathbb{R}}
\newcommand{\cR}{\mathcal{R}}

\newcommand{\Z}{\mathbb{Z}}
\newcommand{\cT}{\mathcal{T}}

\newcommand{\cS}{\mathcal{S}}
\newcommand{\cB}{\mathcal{B}}
\newcommand{\cE}{\mathcal{E}}
\newcommand{\cI}{\mathcal{I}}
\newcommand{\cC}{\mathcal{C}}
\newcommand{\cG}{\mathcal{G}}

\newcommand{\cP}{\mathcal{P}}
\newcommand{\la}{\lambda}

\newcommand{\be}{\begin{equation}}
\newcommand{\ee}{\end{equation}}

\makeatletter
\@addtoreset{equation}{section}

\newtheorem{thm}{Theorem}[section]
\newtheorem{lemma}[thm]{Lemma}

\newtheorem{cor}[thm]{Corollary}

\newtheorem*{namethm}{Theorem}

\begin{document}
\title{Exactness of the Euclidean algorithm and of the Rauzy induction on the space of interval exchange transformations}
\author{Tomasz Miernowski \and Arnaldo Nogueira}
\date{\today}
\maketitle

{\abstract{The two-dimensional homogeneous Euclidean algorithm is the central motivation for the definition of the classical multidimensional continued fraction algorithms, as Jacobi-Perron, Poincar\'e, Brun and Selmer algorithms. The Rauzy induction, a generalization of the Euclidean algorithm, is a key tool in the study of interval exchange transformations. Both maps are known to be dissipative and ergodic with respect to Lebesgue measure. Here we prove that they are exact.}}

\section{Introduction}
Here we study the dynamics of a class of piecewise linear maps defined in the $n$-dimensional Euclidean space, which are  dissipative and ergodic with respect to Lebesgue measure. The aim of this paper is to prove that they bare a stronger property: they are exact, that is, they satisfy a kind of Kolmogorov $0-1$ law.  Our approach suits the whole class of  homogeneous multidimensional continued fraction algorithms. However, for clearness sake, we have chosen to concentrate on two particular examples: the Euclidean algorithm being the central model for those maps and the Rauzy induction acting on the space of interval exchange transformations.

A nonsingular ergodic map $\cT: X \rightarrow X$ is {\it exact} with respect to a measure $\mu$ if, and only if,  for every positive measure set $\Omega \subset  X$, there exists a positive integer $k$ which depends on  the set $\Omega$, such that the measure of the intersection $\cT^{k+1}(\Omega) \cap \cT^{k}(\Omega)$ is positive. The exactness property of an $n$-dimensional homogeneous algorithm implies that suitable projections of the map also bare this property. In particular, the radial projection of the algorithm on the  $(n-1)$-dimensional simplex will also be exact. Moreover, this approach fits also suitable accelerations of homogeneous algorithms, such as multiplicative Jacobi-Perron, Brun, Selmer or Poincar\'{e} algorithms.

As an example one may consider the radial projection of the Rauzy induction algorithm on the simplex, which is a conservative ergodic map. Zorich \cite{zorich} introduced an acceleration of this projection and proved that the resulting map, denoted by $\mathcal{G}$, admits a finite invariant measure. Later Bufetov \cite{bufetov} showed that the map $\mathcal{G}^2=\mathcal{G} \circ \mathcal{G}$ is exact and used this property to prove a result on the decay of correlations for the map $\mathcal{G}^2$. On the other hand, the map $\mathcal{G}$ is not exact. The domain $\Delta$ on which $\mathcal{G}$ acts splits into two disjoint non trivial sets, namely $\Delta_+$ and $\Delta_-$, such that $\mathcal{G}(\Delta_+)=\Delta_-$ and $\mathcal{G}(\Delta_-)=\Delta_+$. 

In order to illustrate our approach, we begin our analysis with the  Euclidean algorithm defined as follows. Let $\R^2_+=\{(\la_1,\la_2)\in\R^2 : \la_1\geq 0, \la_2\geq 0\}$ and consider the map $\cE:\R^2_+\to\R^2_+$ given by
\be \label{E1} 
\cE(\la_1,\la_2)=\left\{  \begin{array}{l l l} (\la_1-\la_2,\la_2) & \text{, if} & \la_1\geq \la_2, \\ (\la_1,\la_2-\la_1) & \text{, if} & \la_1<\la_2. \end{array}\right.
\ee
Other classical versions of the Euclidean algorithm are defined at the end of Section 4. Since $\cE$ is piecewise linear, we may describe its dynamics in matrix notation. Let  
\be \label{B-matrix}
B_1 = \begin{pmatrix} 1 &1 \\ 0 & 1\end{pmatrix}   \ \ \text{and} \ \ B_2=\begin{pmatrix}1 & 0\\ 1 & 1\end{pmatrix},
\ee
be the two elementary matrices which generate the group $SL(2,\Z)$. The definition (\ref{E1}) may be rewritten as 
\be \label{E2}
\cE: \la=\begin{pmatrix} \la_1 \\ \la_2\end{pmatrix}  \mapsto \left\{  \begin{array}{l l l} B^{-1}_1 \la & \text{, if} & \la_1\geq \la_2, \\ B^{-1}_2 \la & \text{, if} & \la_1<\la_2. \end{array}\right.
\ee
The above expression relates the dynamics  of $\cE$ to the linear action of $SL(2,\Z)$. 

\begin{namethm}[\cite{nogueira}, Proposition 8.1]
Let $\la\in \R^2_+$, then
$$
\cup_{m=0}^{\infty}\cup_{k=0}^{\infty} \cE^{-m}(\{\cE^{k}(\la)\})= \R^2_+ \cap SL(2,\Z)\{ \la\}.
$$ \end{namethm}

In particular, the ergodicity of the action of $SL(2,\Z)$ on $\R^2$ implies the ergodicity of $\cE$ with respect to Lebesgue measure. Here we prove the following.

\begin{thm} \label{euclid} The Euclidean algorithm is exact with respect to Lebesgue measure. \end{thm}

Next we consider the Rauzy induction algorithm which acts on the space of interval exchange transformations. To an $n$-interval exchange one associates a first return map induced on a suitable subinterval, which is itself an exchange of $n$ intervals. The aim of the process is to relate the dynamics of an interval exchange to the dynamics of a class of interval exchange maps. Later it was noticed that the Rauzy induction allows the suspension of an interval exchange and one may define a flow on the resulting surface, which is related to the so-called $Teichmuller$ $flow$. Due to the work of Veech \cite{veech1}, and other, the Rauzy induction became a central tool in the dynamical  study of the interval exchange transformations. 

It is known that the Rauzy induction is dissipative and ergodic \cite{veech2}. Here we prove the following result.

\begin{thm} \label{rauzy} The Rauzy induction algorithm acting on the space of interval exchange transformations is exact with respect to Lebesgue measure.\end{thm}

Since the Rauzy induction on the space of $2$-interval exchanges coincides with the Euclidean algorithm, the statement of Theorem \ref{euclid} may be seen as a particular case of the last theorem. However, we decided to treat it independently, since the presentation of the proof becomes more transparent and the techniques involved extend rather easily to the multidimensional case. 

There is a generalization of the Rauzy induction acting on the space of linear involutions (see Danthony and Nogueira \cite{danthony}). Therefore it is natural to conjecture that this transformation is exact as well. This places the exactness property in the context of measured foliations on orientable surfaces. 
\\

The article is organized in the following way. Section 2 concernes the  exactness property. We prove a technical lemma which contains an exactness criterium for ergodic maps. The criterium is general and suits the class of multidimensional continued fraction algorithms.  In Section 3, we present the main dynamical features of the Euclidean algorithm which are exploited in Section 4 to prove that $\cE$ is an exact map - Theorem \ref{euclid}.  It is known that iterations of $\cE$ generate natural Markov partitions of the cone $ \R^2_+$ into subcones. Basically, we prove that we can extract new partitions whose subcones have distortions as large as we want. At the end of Section 4 we introduce two alternative versions of the Euclidean algorithm which are shown to be exact as well. In Section 5 we apply Theorem \ref{euclid} to show exactness of two maps which are conjugate to the Euclidean algorithm. In Section 6 we define interval exchange transformations and the inducing process for interval exchanges, called Rauzy induction. It may be seen as an algorithm acting on copies of the positive cone of $\R^n$. Section 7 is devoted to the study of Rauzy classes of permutations, in particular we prove that every Rauzy graph has a loop, a central fact exploited in the proof of our main theorem. In Section 8, we prove that the Rauzy induction defines suitable partitions of the positive cone $\R^n_+$ which bare similar properties as those defined by the Euclidean algorithm. In Section 9 we  adapt the argument developed in Section 4 to the multidimensional case in order to prove our main result,  Theorem \ref{rauzy}. In the last section we give examples  of other exact  multidimensional continued fraction algorithms which are adapted to the approach developed in our work.

\section{Exactness criterium}

Let $(X,\cB, \mu)$ be a measure space and let $\cT:X\to X$ be a measurable transformation. The map $\cT$ is said to be {\it ergodic} with respect to $\mu$, if for every $\Omega \in \mathcal{B}$ such that $\cT^{-1}(\Omega)=\Omega$, $\mu(\Omega)=0$ or $\mu(X \backslash \Omega)=0$. The map $\cT$ is said to be {\it nonsingular}, if for $\Omega\in\cB$, $\mu(\cT^{-1}(\Omega))=0$ if, and only if, $\mu(\Omega)=0$. The map $\cT$ is said to be {\it exact}, if $\displaystyle \Omega \in \cap_{m\geq 1} \cT^{-m}(\mathcal{B})$ implies $\mu(\Omega)=0$ or $\mu(X \backslash \Omega)=0$. For more information about exact maps one may refer to \cite[Section 1.2]{aaronson}. The  measurable map $\cT$ is said to be {\it bi-measurable}, if, for every $\Omega \in \mathcal{B}$, we have $\cT(\Omega)\in \mathcal{B}$. 

In order to study exactness we introduce an additional dynamical property:   the bi-measurable map $\cT$ satisfies the {\it intersection property}, with respect to the measure $\mu$ if, for every $\Omega \in \mathcal{B}$ with $\mu(\Omega)>0$, there exists $k=k(\Omega)\geq 1$ such that $\mu(\cT^{k+1} (\Omega) \cap \cT^{k} (\Omega)) > 0$.

The next technical lemma whose proof is adapted from \cite[p.11]{rohlin} (see also \cite{bruin-hawkins}) establishes the equivalence between exactness and the intersection property for nonsingular ergodic bi-measurable maps.

\begin{lemma}  \label{exact-lem} Let $(X,\mathcal{B},\mu)$ be a measure space and let $\cT:X\to X$ be bi-measurable, nonsingular and ergodic. Then  $\cT$ is exact if, and only if, it satisfies the intersection property.\end{lemma}

\begin{proof} Assume that $\cT$ is exact. Let $\Omega \in \mathcal{B}$ be of positive measure and consider the following nested sequence of subsets of $X$:
 $$
 \cT(\Omega) \subset \cT^{-1}(\cT(\cT(\Omega)))\subset \ldots  \subset  \cT^{-k}(\cT^{k}(\cT(\Omega))) \subset  \ldots \; .
$$
Let $S=\cup_{k\geq 0} \cT^{-k}(\cT^k(\cT(\Omega)))$. We have 
$$
S=\cup_{k\geq m} \cT^{-k}(\cT^k(\cT(\Omega)))=\cT^{-m}(\cT^{m}(S)), \; \mbox{  for every } \; m \geq 0.
$$ 
Therefore $S \in \cap_{m\geq0}\cT^{-m}(\mathcal{B})$. Moreover, since $\cT(\Omega)\subset S$, by the nonsingularity of $\cT$ we get $\mu(S)>0$. The exactness of $\cT$ implies that $S$ is of full measure in $X$. Since $\mu(\Omega)>0$ there exists $k\geq 1$ such that $\mu(\cT^{-k}(\cT^k(\cT(\Omega))\cap \Omega)>0$. Again, by the nonsingularity of $\cT$, we get
$$\mu(\cT^{k+1}(\Omega)\cap \cT^k(\Omega))>0.$$

Conversely, assume that $\cT$ satisfies the intersection property. Let  $\Omega \in \cap_{m\geq 1} \cT^{-m}(\mathcal{B})$, which is equivalent to $\Omega=\cT^{-m}(\cT^m(\Omega))$ for all $m\geq 1$. Assume that $\mu(\Omega)>0$. In order to show that $\cT$ is exact we have to show that $\Omega$ is of full measure. 

Let $\Lambda=\Omega\setminus \cT(\Omega)$. We have $\Lambda \cap \cT(\Lambda)=\emptyset$. Moreover,  for every $n\geq 1$,
$$
\Lambda=\cT^{-n}(\cT^n(\Omega))\setminus \cT^{-n}(\cT^n(\cT(\Omega)))=\cT^{-n}(\cT^{n}(\Omega)\setminus \cT^{n+1}(\Omega)),$$
we obtain $\cT^n(\Lambda)\cap \cT^{n+1}(\Lambda)=\emptyset$ which implies $\mu(\Lambda)=0$ by the intersection property. We have $\Omega\subset \cT(\Omega)$ up to a null measure set. The same argument applied to $\Lambda'=\cT(\Omega)\setminus \Omega$ gives $\cT(\Omega)\subset \Omega$ which implies $\Omega=\cT(\Omega)$ up to a null measure set. We may write
$$
\cT^{-1}(\Omega)=\cT^{-1}(\cT(\Omega))=\Omega.
$$ 
Since $\cT$ is ergodic and $\Omega$ a positive measure set, we get that it is a full measure set. The map $\cT$ is thus exact.
\end{proof}

In what follows the measure space $(X,\cB,\mu)$ is essentially a positive cone of an Euclidean space $\R^n$ with Borel $\sigma$-algebra and Lebesgue measure $\mu$. The maps considered (the Euclidean algorithm and Rauzy induction) are bi-measurable and nonsingular. Although they do not preserve Lebesgue measure,  they admit invariant measures absolutely continuous with respect to $\mu$.

\section{Euclidean algorithm} 

In this section we recall some results about the dynamics of the Euclidean algorithm that will be used to prove Theorem \ref{euclid}. For more details one may refer to \cite{arnaldo-bb} and references therein. 

Let $\cE: \R^2_+\to \R^2_+$ be the Euclidean algorithm (\ref{E2}). To a point $\la\in\R^2_+$ we may associate a $\{B_1,B_2\}$-valued  (\ref{B-matrix}) sequence of matrices $(B_{m_k})_{k\geq 1}$ such that 
\begin{equation} 
\label{E} \cE^k : \begin{pmatrix} \la_1 \\ \la_2 \end{pmatrix} \mapsto  B^{-1}_{m_k} \cdot\ldots\cdot B^{-1}_{m_1} \begin{pmatrix} \la_1 \\ \la_2 \end{pmatrix}  \ \ \text{for every} \ k\geq 1.
\end{equation}
One may easily verify that (\ref{E}) holds if and only if 
\begin{equation} \label{cone} \begin{pmatrix} \la_1 \\ \la_2\end{pmatrix} \in B_{m_1}\cdot\ldots\cdot B_{m_k} (\R^2_+) \ \ \text{for every} \ k\geq 1. \end{equation}  

The sequence $(B_{m_k})$ is called the {\it expansion} of the point $\la$. This comes from the fact that it is closely related to the continued fraction expansion of the ratio $ \la_1/ \la_2$. Assume this ratio irrational. In such case, the sequence $(B_{m_k})$ contains infinitely many of both matrices $B_1$ and $B_2$. We may define a sequence of integers $(a_i)_{i\geq 0}$ as follows. Let $a_0=0$ if the sequence begins with $B_1$, otherwise let $a_0$ be the number of consecutive matrices $B_2$ at the beginning of the sequence. Next, let $a_1$ be the number of consecutive matrices $B_1$ that follow. Then, define $a_2$ to be the number of consecutive matrices $B_2$ that come next, and so on:
$$(B_{m_k}) = \underbrace{B_2\ldots B_2}_{a_0}  \underbrace{B_1\ldots B_1}_{a_1} \underbrace{B_2\ldots B_2}_{a_2} \underbrace{B_1\ldots B_1}_{a_3}\ldots$$ 
It may be shown that     
$$
\frac{\la_1}{\la_2} =  [a_0;a_1,a_2, \ldots]=a_0+\cfrac{1}{a_1+\cfrac{1}{a_2+\cfrac{1}{\ddots}}} \ .
$$

\subsection{Partitions of $\R^2_+$} 

From (\ref{cone}) we deduce that for every $k\geq 1$, the positive cone $\R^2_+$ is decomposed into $2^k$ subcones, of disjoint nonempty interiors, which correspond to different sequences of elementary matrices involved in the iterations of $\cE$. Let $\mathcal{P}^{(k)}$ be the $k$th partition
\begin{equation} \label{decomposition} \mathcal{P}^{(k)}=\{  B_{m_1}\cdot\ldots\cdot B_{m_k} (\R^2_+) : m_1,\ldots,m_k\in\{1,2\}\}.\end{equation}
For every $k\geq 1$, the partition $\cP^{(k+1)}$ is a refinement of $\cP^{(k)}$. To be more precise, if $C^{(k)}\in \cP^{(k)}$ is defined by a couple of vectors $(l_1,l_2)$, then it generates two elements of $\cP^{(k+1)}$ defined by the couples $(l_1,l_1+l_2)$ and $(l_1+l_2,l_2)$.

Fix $\la\in\R^2_+$ with irrational ratio $\la_1/\la_2$. For every $k\geq 1$, let $C^{(k)}_{\la}\in\mathcal{P}^{(k)}$ be the cone defined by (\ref{cone}). We obtain a nested sequence of subcones of $\R^2_+$ and one may show that the intersection $\cap_{k\geq 1} C^{(k)}_{\la}$ equals the radial line $\{\alpha\la : \alpha\geq 0\}$. For every $k\geq 1$, let $l_1^{(k)}$ and $l_2^{(k)}$ stand for the two column-vectors of the matrix $B_{m_1}\cdot\ldots\cdot B_{m_k}$, generating the corresponding cone $C^{(k)}_{\la}$. From (\ref{E}) we deduce that $\cE^k(C^{(k)}_{\la})=\R^2_+$ and the vectors $l_1^{(k)},l_2^{(k)}$ are sent onto the canonical basis of $\R^2_+$:
\be \label{base}  \cE^k(l_i^{(k)})=e_i, \ i=1,2. \ee 

 The next lemma shows a central property of the family of partitions  $(\mathcal{P}^{(k)})_{k\geq 1}$  of $\R^2_+$. Namely, the angles formed by the column-vectors  $ l_1^{(k)}, l_2^{(k)}$ which define the subcones $C^{(k)}$ go to zero uniformly as $k\to\infty$.

\begin{lemma} \label{compact} 
Let $K$ be a compact measurable subset of $\R^2_+$. Then
$$ 
\lim_{k\to\infty} \max_{C\in\cP^{(k)}} \mu(K\cap C)=0. 
$$
\end{lemma} 
\begin{proof} 
Let $K$ be a compact subset of $\R^2_+$. Let $\alpha>0$ be such that 
$K \subset \Delta(\alpha)= \{ \la \in \R^2_+: \la_1+\la_2\leq \alpha\}$. Let $C\in\cP^{(k)}$ and $l_1, l_2$ be the column-vectors of the matrix which defines $C$.  A trivial computation gives 
$$
\mu(\Delta(\alpha)\cap C)=\frac{1}{2}\frac{\alpha^2}{\Vert l_1\Vert_1\Vert l_2\Vert_1}
$$
which implies that
$$
\mu(K\cap C)\leq \frac{1}{2}\frac{\alpha^2}{\Vert l_1\Vert_1\Vert l_2\Vert_1}\leq\frac{\alpha^2}{2(k+1)}.
$$ 
We conclude that
$$
\lim_{k\to\infty} \max_{C\in\cP^{(k)}} \mu(K\cap C)=0. 
$$
\end{proof}

Next, let $\Vert \cdot \Vert$ stand for the Euclidean norm. We have the following result which comes from the continued fraction expansion interpretation of $\cE$.
 
\begin{thm}[\cite{arnaldo-bb}, Section 4]  \label{cones} 
Let $\la\in\R^2_+$ with irrational ratio $\la_1/ \la_2$ and $ l_1^{(k)}, l_2^{(k)}$ be the column-vectors which define $C^{(k)}_{\la}$, $k\geq 1$. Then  the following properties hold: 
\begin{enumerate}
\item There exist infinitely many integers $k\geq1$ such that 
$$
\frac{1}{\theta}\leq \frac{\Vert l_1^{(k)}\Vert}{\Vert l_2^{(k)}\Vert}\leq \theta,
$$  
where $\theta>1$ is a constant independent of $\la$.
\item For every $N>0$, for almost every $\la$ there exist infinitely many integers $s\geq 1$ such that 
$$ \frac{\Vert l_1^{(s)}\Vert}{\Vert l_2^{(s)}\Vert} > N \ \ \text{or} \ \ \frac{\Vert l_2^{(s)}\Vert}{\Vert l_1^{(s)}\Vert} > N.$$
\end{enumerate}
\end{thm}

Next we introduce the notion of distortion of a subcone which will be needed in the next section. Let $C\in\cP^{(k)}$ and $l_1, l_2$ be the column-vectors of the matrix which defines $C$. We call the $distortion$ of $C$ the number
$$
\max\left\{ \frac{\Vert l_1 \Vert}{\Vert l_2\Vert},\frac{\Vert l_2\Vert}{\Vert l_1\Vert}\right\}.
$$ 
The second part of the last theorem implies that there exists a partition of $\R^2_+$ formed only by subcones of $\cup_{ k\geq 1} \mathcal{P}^{(k)}$ whose  distortions are as large as we want.

\begin{cor} \label{distortionthm}
Let $N>1$. Then there exists a partition $\mathcal{P}_N$ of $\R^2_+$ formed by subcones of $\cup_{ k\geq 1} \mathcal{P}^{(k)}$ such that, for every $C\in \mathcal{P}_N$, its distortion is greater than $N$.
\end{cor} 
\begin{proof} From Theorem \ref{cones} we know that for almost every $\la\in\R^2_+$ there exist infinitely many integers $s\geq 1$ such that $C^{(s)}_{\la}$ is a distorted cone. Let $s(\la)$ be the smallest of those integers and define $\cP_N$ to be the collection of all cones $C^{(s(\la))}_{\la}$ obtained this way. We claim that $\cP_N$ is a partition of $\R^2_+$. 

Clearly, the union of all the members of $\cP_N$ covers $\R^2_+$ up to a null measure set. Moreover, let  $C_1$ and $C_2$ be subcones  belonging to the collection $ \cP_N$. We claim that $C_1=C_2$, otherwise $C_1\cap C_2$ is a null measure set. 

Assume $\mu(C_1\cap C_2)>0$. There exist $\la^1$ and $\la^2$ such that $C_i=C^{(s(\la^i))}_{\la^i}, \ i=1,2$. Assume $s_1=s(\la^1)\geq s_2= s(\la^2)$. 
Since the partition $\cP^{(s_1)}$ is a refinement of the partition $\cP^{(s_2)}$ and $C_1$ and $C_2$ have a non trivial intersection, the cone $C_1$ must be contained in $C_2$. This implies $\la^1\in C_2$ and thus $s_1=s_2$ by the definition of $s(\la^1)$. This implies $C_1=C_2$ and proves that $ \cP_N$ is a partition of $\R^2_+$ with the desired properties.
\end{proof}

\section{Proof of Theorem \ref{euclid}}

To prove that the Euclidean algorithm is exact with respect to Lebesgue measure we use Lemma \ref{exact-lem}. The main steps of the proof are the following. 

Let $\Omega \subset \R^2_+$ be a positive Lebesgue measure set and $\la^0$ a density point of $\Omega$. First, we construct a sequence $(Q_n)_{n\geq 1}$ of quadrilateral domains that shrink to $\{\la^0 \}$ as $n\to\infty$. Using a version of Lebesgue density theorem, we show that given $\varepsilon>0$, for every $n$ sufficiently large we have 
\begin{equation} \label{dense} \mu(\Omega \cap Q_n)\geq (1-\varepsilon)\mu(Q_n), \end{equation} where $\mu$ stands for Lebesgue measure. 

Fix $Q_n$ satisfying (\ref{dense}). We will show that a similar density property holds for a smaller quadrilateral $Q$ which is the intersection of $Q_n$ with a subcone coming from a ``distorted'' partition $\cP_N$ given by Corollary \ref{distortionthm}.  

This new quadrilateral $Q$ is the intersection of $Q_n$ with a cone $C^{(s)}$ defined by (\ref{cone}). We know that $\cE^s(C^{(s)})=\R^2_+$, which implies that the vertices of the quadrilateral $\cE^s(Q)$ are fixed by $\cE$. Moreover, it still satisfies a density condition close to (\ref{dense}). Finally we use the distortion property of the cone $C^{(s)}$ to show that the intersection $\cE^s(Q)\cap \cE^{s+1}(Q)$ is large enough to imply $\mu(\cE^{s}(\Omega)\cap \cE^{s+1}(\Omega))>0$.    

The key point of the proof relies on the fact that the value of $N$ can be chosen independently of the values of $\varepsilon$ and $n$.

\subsection{The first sequence of quadrilaterals} 

Let $\Omega \subset \R^2_+$ be a set of positive Lebesgue measure and fix $\la^0=(\la_1^0,\la_2^0)$ a density point of $\Omega$, which is an interior point of $\R^2_+$. For every $n\geq 1$, let $Q_n$ be the quadrilateral (trapezoid) $p_n q_n r_n s_n$ whose vertices are defined as follows (see Figure 1):

\begin{itemize}
\item $p_n=\la^0-\frac{1}{2n}(\la_2^0,-\la_1^0)=(\la_1^0 - \frac{1}{2n}\la_2^0,\la_2^0+\frac{1}{2n}\la_1^0),$
\item $q_n=\la^0+\frac{1}{2n}(\la_2^0,-\la_1^0)=(\la_1^0 + \frac{1}{2n}\la_2^0,\la_2^0-\frac{1}{2n}\la_1^0),$
\item $r_n=(1+\frac{1}{n})q_n$ and  $s_n=(1+\frac{1}{n})p_n$.
\end{itemize}
The point $\la^0$ is the middle point of the segment $p_n q_n$ and the Euclidean distance between the parallel segments $p_n q_n$ and $s_n r_n$ is equal to the length of $p_n q_n$, that is $\frac{1}{n}\Vert \la^0 \Vert$. The nested sequence of quadrilaterals $(Q_n)_{n\geq 1}$ satisfies $\cap_{n\geq 1} Q_n=\{\la^0 \}$. Moreover, for $n$ large enough the quadrilateral $Q_n$ is contained in $\R^2_+$.

\begin{figure}

\begin{picture}(300,220)(0,5)

\put(120 ,35){\mbox{\includegraphics[scale=0.5]{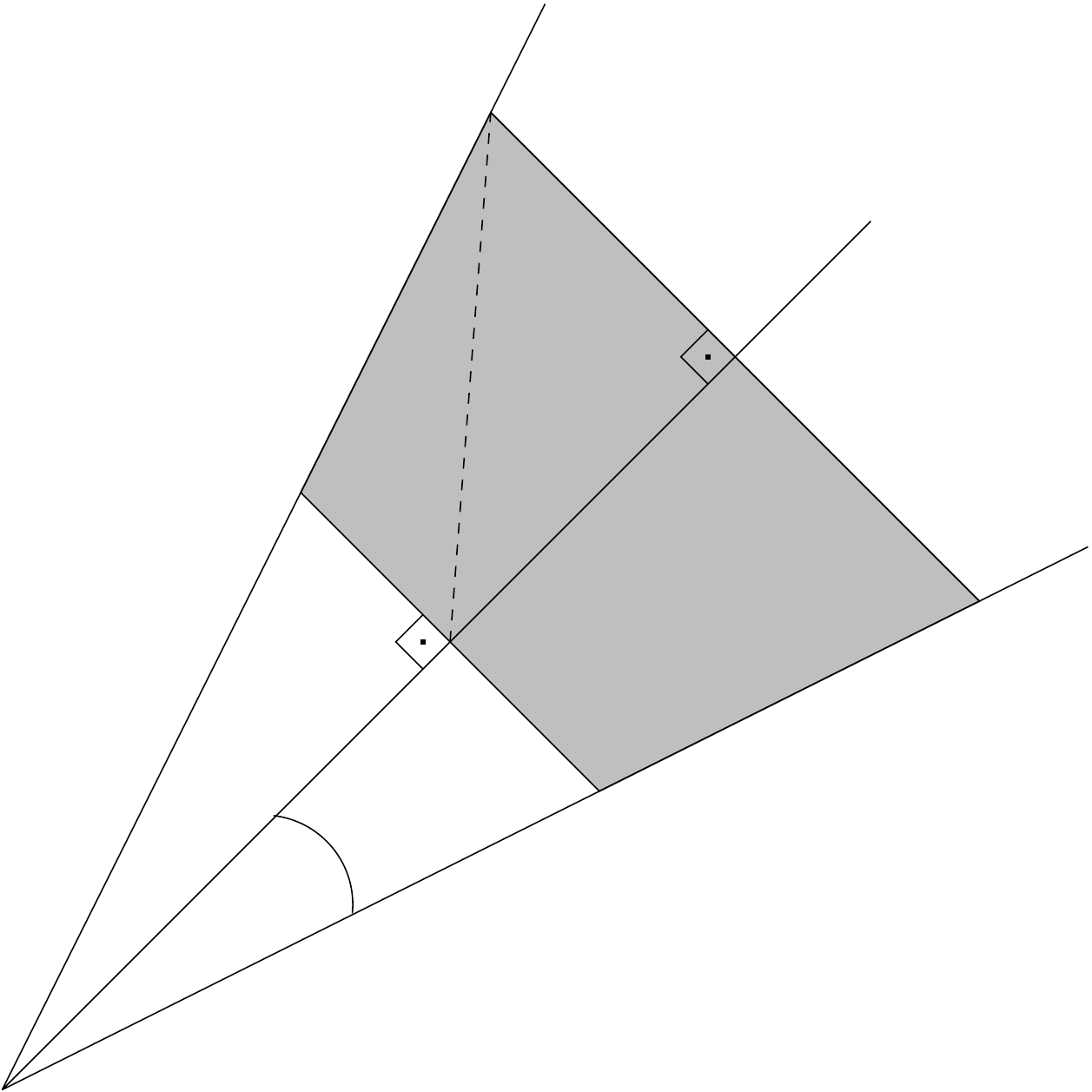}}}

\put(110,30){\mbox{$0$}}

\put(200,107){\mbox{$\la^0$}}
\put(153,135){\mbox{$p_n$}}
\put(220,75){\mbox{$q_n$}}
\put(283,107){\mbox{$r_n$}}
\put(185,200){\mbox{$s_n$}}

\put(157,65){\mbox{$\phi_n$}}

\put(202, 150){\mbox{$\rho_n$}}

\put(140,5){\mbox{Figure 1. Quadrilateral $Q_n$.}}
\end{picture}
\end{figure}

\begin{lemma} 
For every $n\geq 1$ there exists a ball $B(\la^0,\rho_n)$ centered at $\la^0$ of radius $\rho_n$, such that $Q_n\subset B(\la^0, \rho_n)$, $\rho_n\to 0$ as $n\to\infty$ and 
 \begin{equation} 
 \label{ball} \lim_{n\to\infty} \frac{ \mu (Q_n)}{\mu (B(\la^0,\rho_n))} =\frac{4}{5\pi}. 
 \end{equation} \end{lemma}

\begin{proof} We have 
$$
\mu(Q_n)=\frac{2n+1}{2n^3}\Vert \la^0 \Vert^2.
$$
The points of $Q_n$ situated the farthest from $\la^0$ are the vertices $r_n$ and $s_n$. This implies that $Q_n$ is contained in the ball centered at $\la^0$ of radius $\rho_n=\sqrt{5n^2+2n+1}\frac{\Vert \la^0 \Vert}{2n^2}$. We have $\rho_n\to 0$, as $n\to \infty$, and 
$$
\frac{\mu(Q_n)}{\mu(B(\la^0,\rho_n))}=\frac{2n(2n+1)}{\pi (5n^2+2n+1)}\to \frac{4}{5\pi}.
$$ 
\end{proof}

\begin{cor} \label{Qn} For every $\varepsilon >0$ there exists $n\geq 1$ such that $\mu(\Omega \cap Q_n)\geq (1-\varepsilon) \mu(Q_n)$.  \end{cor}

\begin{proof} Let $B(\la^0,\rho_n)$ be the sequence of balls defined in the previous lemma. Since $\la^0$ is a density point of $\Omega$, we know that 
$$
\lim_{n\to\infty} \frac{\mu(\Omega \cap B(\la^0,\rho_n))}{\mu(B(\la^0,\rho_n))}=1.
$$  
Fix $\delta \in (0,1)$. For $n$ large enough we have $\mu(\Omega\cap B(\la^0,\rho_n))\geq (1-\delta)\mu(B(\la^0,\rho_n))$. Since $Q_n\subset B(\la^0,\rho_n)$, we get 
$$
\mu(\Omega \cap Q_n)= \mu(\Omega \cap B(\la^0,\rho_n)) - \mu(\Omega \cap (B(\la^0,\rho_n) \setminus Q_n))
$$
which implies 
$$
\mu(\Omega \cap Q_n)\geq (1-\delta) \mu( B(\la^0,\rho_n)) -  \mu(B(\la^0,\rho_n) \setminus Q_n)
= \mu(Q_n)-\delta\mu(B(\la^0,\rho_n)).
$$
 The relation (\ref{ball}) implies
$$\frac{\mu(\Omega\cap Q_n)}{\mu(Q_n)}\geq 1-\delta \frac{\mu(B(\la^0,\rho_n))}{\mu(Q_n)}\geq 1-\delta 2\pi,$$ for $n$ large enough. Since $\delta$ may be chosen as small as we wish, the claim follows.  \end{proof}

\subsection{The distorted quadrilateral}

Let $Q_n$ be a quadrilateral defined above and satisfying (\ref{dense}). In the next lemma we consider the intersection of $Q_n$ with cones of the partition $\cP_N$ given by Corollary \ref{distortionthm}. In particular, for $N$ large enough, one of the new smaller quadrilaterals satisfies a density condition close to (\ref{dense}).     

\begin{lemma} \label{Q} 
For $N\geq 1$ large enough there exists $C\in\cP_N$ such that the quadrilateral  $Q =Q_n\cap C$ satisfies the density condition
$$ 
\mu(\Omega \cap Q)\geq (1-2\varepsilon) \mu(Q).
$$ 
\end{lemma}

\begin{proof} Let $C(Q_n)$ be the subcone of $\R^2_+$ generated by the couple of vectors corresponding to the vertices $p_n$ and $q_n$ of $Q_n$. Thus $C(Q_n)$ is the smallest cone containing the quadrilateral $Q_n$. 
For every $N\geq 1$, the family, by Corollary \ref{distortionthm}, the family $\mathcal{P}_N$ is a partition of $\R^2_+$ and 
$$ \lim_{N\to\infty} \max_{C\in\cP_N} \mu(Q_n\cap C)=0, $$ by Lemma \ref{compact}. This means that for $N$ large enough we get 
$$
\sum_{\substack{C\in\mathcal{P}_N \\ C\subset C(Q_n)}} \mu(Q_n\cap C)\geq (1-\varepsilon) \mu(Q_n).
$$ 
If $\mu(\Omega \cap Q_n\cap C) < (1-2\varepsilon ) \mu (Q_n\cap C)$ for every cone $C$ in the above sum, we get
$$
\mu(\Omega \cap Q_n)< (1-2\varepsilon)\mu(Q_n) + \varepsilon\mu(Q_n) = (1-\varepsilon)\mu(Q_n),  
$$ 
which contradicts the choice of $Q_n$.\end{proof}

Let $N\geq 1$ be fixed and let $Q$ be a quadrilateral given by the previous lemma. It is defined as $Q=Q_n\cap C^{(s)}$, where $C^{(s)}\in \cP_N\cap \cP^{(s)}$ for some $s\geq 1$. Let also $l_1^{(s)}$ and $l_2^{(s)}$ be the column-vectors of the matrix $B_{m_1}\cdots B_{m_s}$ generating $C^{(s)}$. Since $C^{(s)}\in \cP_N$, those vectors satisfy the corresponding distortion condition. Without loss of generality, assume $\Vert l_1\Vert \geq N \Vert l_2 \Vert$.

The vertices of $Q$ may be written as $\alpha_1 l_1, \beta_1 l_1, \beta_2 l_2$ and $\alpha_2 l_2$, where $0<\alpha_1<\beta_1$ and $0<\alpha_2<\beta_2$ (see Figure 2). Let $\phi_n$ be the angle between the vectors corresponding to the points $q_n$ and $\la^0$ of the quadrilateral $Q_n$ (see Figure 1). We have the following estimates.

\begin{figure}

\begin{picture}(300,260)(0,5)

\put(20 ,35){\mbox{\includegraphics[scale=0.5]{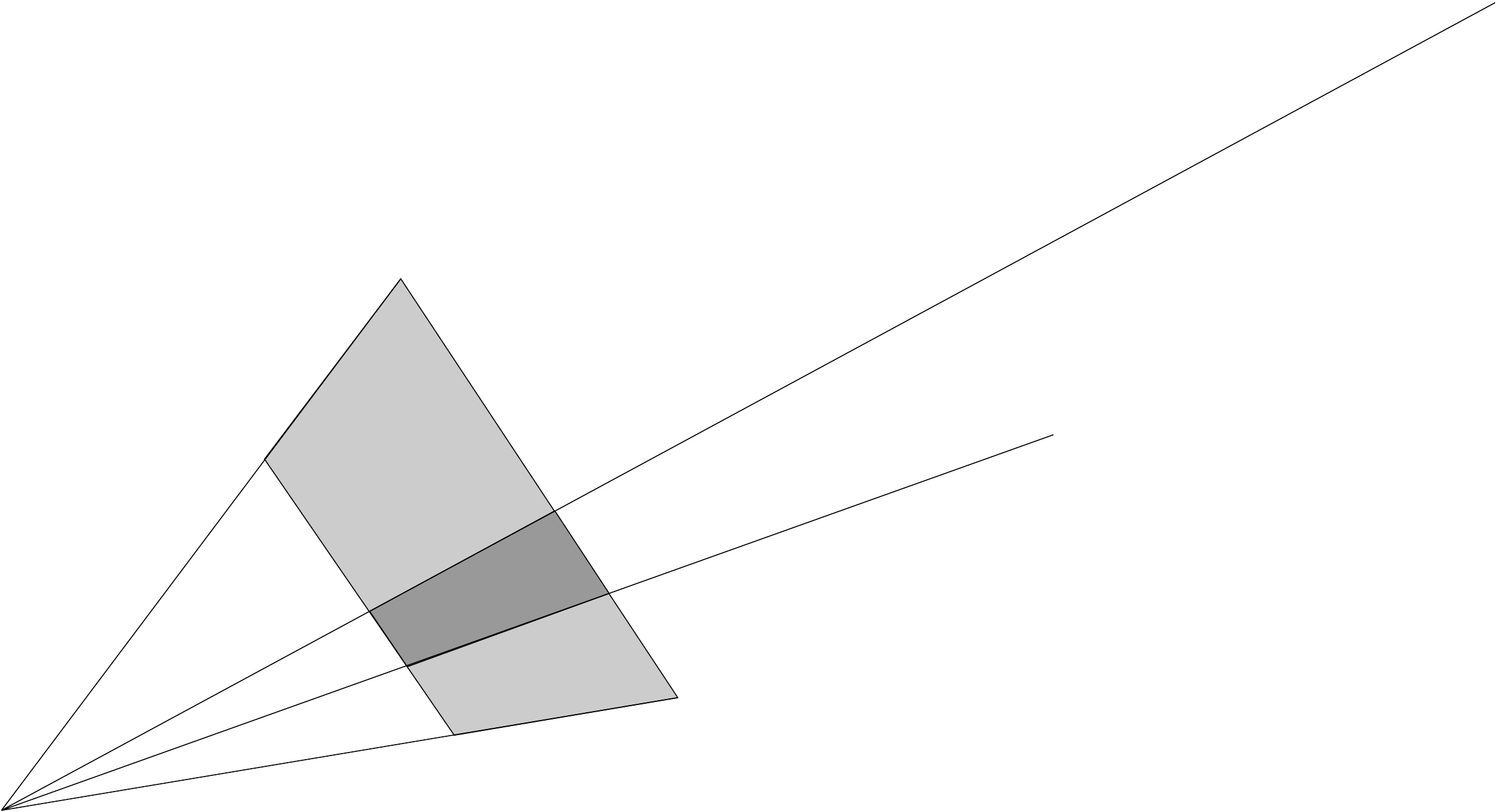}}}

\put(120,5){\mbox{Figure 2. Quadrilaterals $Q_n$ and $Q$.}}

\put(125,44){\mbox{$q_n$}}
\put(180,52){\mbox{$r_n$}}

\put(70,120){\mbox{$p_n$}}
\put(105,165){\mbox{$s_n$}}

\put(275,115){\mbox{$l_2$}}
\put(365,235){\mbox{$l_1$}}

\put(127,83){\mbox{$Q$}}
\put(105,115){\mbox{$Q_n$}}

\put(98,57){\mbox{$\alpha_2 l_2$}}
\put(173,81){\mbox{$\beta_2 l_2$}}

\put(82,83){\mbox{$\alpha_1 l_1$}}
\put(150,118){\mbox{$\beta_1 l_1$}}

\put(15,25){\mbox{$0$}}
\end{picture}

\end{figure}

\begin{lemma} \label{ratios} The vertices of $Q$ satisfy 
$$ \frac{\alpha_1}{\beta_1}=\frac{\alpha_2}{\beta_2}=\frac{n}{n+1} \ \ \text{and} \ \ \frac{\alpha_2}{\alpha_1}=\frac{\beta_2}{\beta_1}\geq N\cos\phi_n.$$  \end{lemma}

\begin{proof} The first equality above comes from the fact that $Q$ is a trapezoid and from the definition of $Q_n$. To prove the second one, recall that $Q$ is the intersection of the trapezoid $Q_n$ with a subcone of the initial cone $C(Q_n)$. This means, according to Figures 1 and 2, that $\alpha_2\Vert l_2\Vert\geq \Vert \la^0 \Vert$ and $\alpha_1 \Vert l_1\Vert \leq \Vert q_n\Vert$. We get
$$\frac{\alpha_2}{\alpha_1}\geq \frac{\Vert l_1\Vert \Vert \la^0\Vert}{\Vert l_2\Vert \Vert q_n \Vert}\geq N\cos\phi_n.$$ The same argument holds for $\beta_2 / \beta_1$.  \end{proof}

\subsection{Intersection property} 
We know that $\cE^s(C^{(s)})=\R^2_+$ and $\cE^s(l_i)=e_i$, $i=1,2$ (see (\ref{base})). Since the vertices of $Q$ are $\alpha_1l_1,\beta_1 l_1, \beta_2 l_2$ and $\alpha_2 l_2$, the image $\cE^s(Q)$ is the trapezoid $T$ of vertices $(\alpha_1,0), (\beta_1,0), (0,\beta_2)$ and $(0,\alpha_2)$ (see Figure 3). Moreover, the map $\cE^s$ restricted to $C^{(s)}$ is bijective and preserves Lebesgue measure, which implies 
$$\mu(\cE^s(\Omega)\cap T)\geq (1-2\varepsilon)\mu(T).$$
 
Consider the smaller trapezoid $T_+=\{\la\in T: \la_2\geq \la_1\}$ whose vertices are $(0,\alpha_2)$, $(0,\beta_2)$, $(\frac{\beta_1\beta_2}{\beta_1+\beta_2}, \frac{\beta_1\beta_2}{\beta_1+\beta_2})$ and $(\frac{\alpha_1\alpha_2}{\alpha_1+\alpha_2}, \frac{\alpha_1\alpha_2}{\alpha_1+\alpha_2})$. A short calculation and Lemma \ref{ratios} give 
\be \label{?} \frac{\mu(T_+)}{\mu(T)}=\left(\frac{\beta_1(\beta_2)^2}{\beta_1+\beta_2} - \frac{\alpha_1(\alpha_2)^2}{\alpha_1+\alpha_2}\right)\frac{1}{\beta_1\beta_2-\alpha_1\alpha_2} = \frac{\beta_2}{\beta_1+\beta_2} \geq \frac{N\cos\phi_n}{N\cos\phi_n+1}.\ee 
The choice of $N$ is independent of the choice of $\varepsilon$ and of the first quadrilateral $Q_n$. Choosing $N$ large enough, we may assume $\mu(T_+)\geq (1-\varepsilon) \mu(T)$ which implies
\begin{equation} \label{denseT+} \mu(\cE^s(\Omega)\cap T_+)\geq (1-3\varepsilon) \mu(T_+).\end{equation}   

\begin{figure}

\begin{picture}(300,320)(0,5)

\put(40 ,40){\mbox{\includegraphics[scale=0.5]{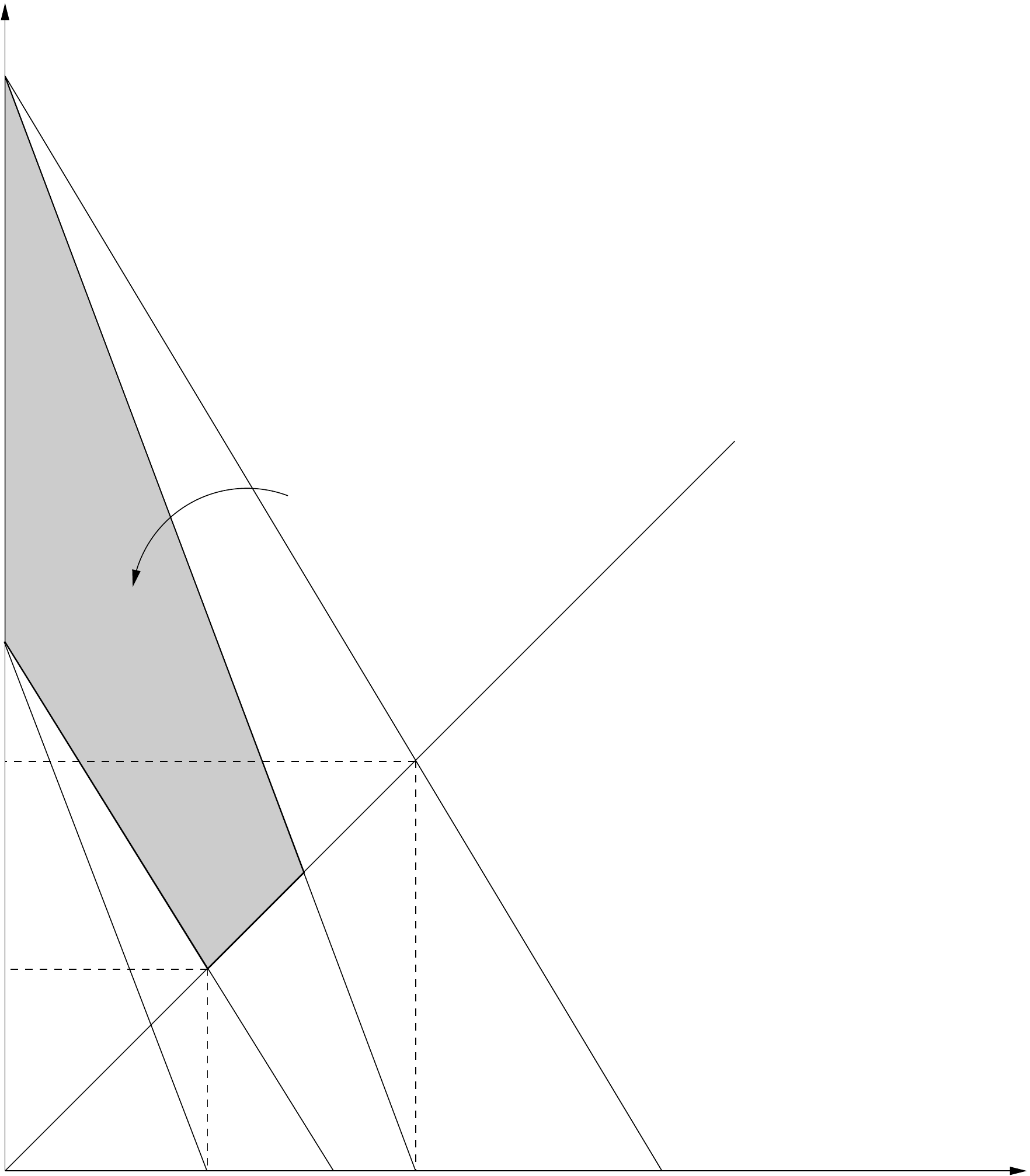}}}

\put(120,5){\mbox{Figure 3. Trapezoids $T$, $T_+$ and $\cE(T_+)$.}}

\put(113,30){\mbox{$\alpha_1$}}
\put(200,30){\mbox{$\beta_1$}}
\put(130,27){\mbox{$\frac{\beta_1\beta_2}{\beta_1+\beta_2}$}}
\put(75,27){\mbox{$\frac{\alpha_1\alpha_2}{\alpha_1+\alpha_2}$}}

\put(25,175){\mbox{$\alpha_2$}}
\put(25,310){\mbox{$\beta_2$}}
\put(10,140){\mbox{$\frac{\beta_1\beta_2}{\beta_1+\beta_2}$}}
\put(10,90){\mbox{$\frac{\alpha_1\alpha_2}{\alpha_1+\alpha_2}$}}

\put(115,205){\mbox{$T_+\cap \cE(T_+)$}}

\put(30,30){\mbox{$0$}}
\end{picture}
\end{figure}

\begin{lemma} \label{intersectionlem} For $N$ large enough  
\begin{equation} \label{intersection} \mu(\cE(T_+)\cap T_+) \geq \frac{1}{2} \mu (T_+)=\frac{1}{2}\mu(\cE(T_+)).\end{equation} \end{lemma}

\begin{proof} The image $\cE(T_+)$ is the trapezoid of vertices $(0,\alpha_2)$, $(0,\beta_2)$, $(\frac{\beta_1\beta_2}{\beta_1+\beta_2}, 0)$, and  $(\frac{\alpha_1\alpha_2}{\alpha_1+\alpha_2}, 0)$. We will show that for $N$ large enough we get
$$ 
 \frac{\beta_1\beta_2}{\beta_1+\beta_2}\geq \frac{\beta_1+\alpha_1}{2}
$$
which implies (\ref{intersection}). The above inequality is equivalent to
$$\beta_1\beta_2\geq (\beta_1)^2+\alpha_1\beta_1+\alpha_1\beta_2,$$
$$\frac{\beta_2}{\beta_1}\geq 1+\frac{\alpha_1}{\beta_1}+\frac{\alpha_1}{\beta_1}\frac{\beta_2}{\beta_1},$$
$$\frac{\beta_2}{\beta_1} \geq 2n+1,$$ by Lemma \ref{ratios}. Since the choice of $N$ is independent of the choice of the first quadrilateral $Q_n$ and  since $\beta_2/\beta_1\geq N\cos\phi_n$, the claim follows. \end{proof}

Now we conclude the proof of Theorem \ref{euclid}. Let $\varepsilon<1/12$ and $N$ be large enough for (\ref{denseT+}) and (\ref{intersection}) to hold. We get 
$$\mu(\cE^s(\Omega)\cap T_+\cap \cE(T_+))> \frac{1}{2} \mu(T_+\cap \cE(T_+)) \ \ \text{and} \ \  \mu(\cE^{s+1}(\Omega)\cap T_+\cap \cE(T_+))> \frac{1}{2} \mu(T_+\cap \cE(T_+)) $$ which imply $\mu(\cE^s(\Omega)\cap \cE^{s+1}(\Omega))>0$. Therefore $\cE$ is exact with respect to Lebesgue measure and the proof of Theorem \ref{euclid} is complete. \hfill $\square$

\subsection{Other versions of $\cE$} 
In the literature one may find other definitions of the Euclidean algorithm. We would like to mention two of them. The first one, denoted by $\cE_{\sigma}$, is defined on $\R^2_+$ by
$$\mathcal{E}_{\sigma}(\la_1,\la_2)=(\la_{\sigma(1)},\la_{\sigma(2)}-\la_{\sigma(1)})=\left\{ \begin{array}{lll} (\la_1,\la_2-\la_1) & \text{if} & \la_1\leq \la_2, \\ (\la_2,\la_1-\la_2) & \text{if} & \la_2 < \la_1. \end{array} \right.$$
Here $\sigma$ stands for the permutation (depending on $\la$) of the set $\{1,2\}$ such that $\la_{\sigma(1)}\leq \la_{\sigma(2)}$. It differs from $\cE$ only by the permutation of coordinates in the upper sub-cone $\{\la_2\geq \la_1\}$. 

The other one is a projection of $\cE$ onto the subset $\Lambda^2=\{\la\in\R^2_+: \la_1\leq \la_2\}$. We define it as follows.
$$\mathcal{E}_{\pi}(\la_1,\la_2)=\pi(\la_1,\la_2-\la_1) =\left\{ \begin{array}{lll} (\la_1,\la_2-\la_1) & \text{if} & \la_1\leq \la_2-\la_1 \\ (\la_2-\la_1,\la_1) & \text{if} & \la_2-\la_1 < \la_1, \end{array} \right.$$ where $\pi$ stands for the permutation of coordinates (depending on $\la$) arranging them in nondecreasing order. 

From the ergodicity of $\cE$ one may deduce the ergodicity of $\cE_{\sigma}$ and $\cE_{\pi}$. The proof of Theorem \ref{euclid} presented above applies with minor changes to those two transformations. 

\begin{cor} \label{sig-pi} The maps $\cE_{\sigma}$ and $\cE_{\pi}$ are exact with respect to Lebesgue measure. \end{cor}

\section{Algorithms conjugate to $\mathcal{E}_{\sigma}$ and $\cE_{\pi}$}

In this section we present two examples of transformations which are conjugate to the Euclidean algorithm and inherit therefore the exactness property. The first one is a normalization of the three-dimensional Poincar\'{e} algorithm. The other one is an example of a so-called fully subtractive algorithm introduced by Schweiger \cite[Chapter 9]{schweiger}.  

\subsection{The Poincar\'{e} algorithm} 
For every point $\la\in\R^3_+$ let $\sigma$ be a permutation of the set $\{1,2,3\}$ such that $\la_{\sigma(1)} \leq  \la_{\sigma(2)}\leq  \la_{\sigma(3)}$. The three-dimensional Poincar\'{e} algorithm is the map $\cP:\R^3_+\to\R^3_+$ defined by 
$$
\cP(\la_1,\la_2,\la_3) = (\la_{\sigma(1)},\la_{\sigma(2)}-\la_{\sigma(1)},\la_{\sigma(3)}-\la_{\sigma(2)}).
$$
In \cite[Theorem 2.1]{nogueira} it is shown that $\cP$ is not ergodic with respect to Lebesgue measure. In fact $\cP$ admits a nontrivial absorbing set $\Gamma\subset \R^3_+$ and its restriction to this set is conjugate to an extension of the Euclidean algorithm $\cE_{\sigma}$. To be more precise, the map $\cP:\Gamma\to \Gamma$ is conjugate to $\cE_{\sigma}\times id : \R^3_+\to \R^3_+$. 

The nonergodic transformation $\cP$ cannot be exact. However, being conjugate to $\cE_{\sigma}\times id$, it satisfies the intersection property. Now consider the projection of the algorithm $\cP$ onto the two-dimensional simplex $\Delta_2=\{\lambda\in\R^3_+: \la_1+\la_2+\la_3=1\}$ given by 
$$\tilde{\cP}:\Delta_2 \ni \la \mapsto \frac{\cP(\la)}{\Vert \cP( \la) \Vert_1}\in\Delta_2.$$ This normalization of $\cP$, called {\it Daniels-Parry map} \cite[p.185]{schweiger}, is ergodic with respect to Lebesgue measure on $\Delta_2$ \cite[Theorem 2.3]{nogueira}. From the intersection property of $\cP$ we deduce the analogous property for $\tilde{\cP}$. 

\begin{cor} The normalized Poincar\'{e} algorithm $\tilde{\cP}:\Delta_2\to\Delta_2$ is exact with respect to Lebesgue measure. \end{cor}

\subsection{A fully subtractive transformation} 
Let  the map $\cS:\{\la \in\R^3_+ : \la_1\leq \la_2\leq \la_3\} \to \{\la \in\R^3_+ : \la_1\leq \la_2\leq \la_3\}$ be defined by
$$ 
\cS (\la_1,\la_2,\la_3)= \pi(\la_1,\la_2-\la_1,\la_3-\la_1),
$$
where $\pi$ is a permutation arranging the coordinates in the nondecreasing order. In an earlier paper \cite{nogueira-miernowski} we showed that $\cS$ is not ergodic with respect to Lebesgue measure. However, one may normalize $\cS$ by imposing the last coordinate to be equal to one. We obtain a new transformation 
$$ 
\tilde{\cS} (\la_1,\la_2) =  \left\{ \begin{array}{lll}  \frac{1}{1-\la_1}(\la_1,\la_2-\la_1) & \text{if} & \la_1 \leq \la_2-\la_1, \\     \frac{1}{1-\la_1}(\la_2-\la_1,\la_1) & \text{if} & \la_2-\la_1 < \la_1\leq 1-\la_1, \\  \frac{1}{\la_1}(\la_2-\la_1,1-\la_1) & \text{if} & 1-\la_1 < \la_1,      \end{array} \right.  
$$
defined on the set $D=\{\la\in\R^2_+: 0\leq \la_1\leq \la_2\leq 1\}$. The map $\tilde{\cS}$ is ergodic with respect to Lebesgue measure \cite[Theorem 1.3]{nogueira-miernowski}. This comes from the fact that $\tilde{\cS}$ admits a nontrivial absorbing set on which its dynamics is conjugate to $\cE_{\pi}$. Now we may improve our earlier result.

\begin{cor} The normalized algorithm $\tilde{\cS}:D\to D$ is exact with respect to Lebesgue measure. \end{cor}

\section{Interval exchange transformations and Rauzy induction} 

Throughout the remaining part of the paper let $n\geq 2$ be a fixed integer. Consider the Euclidean cone
$$
\mathcal{C}=\{\lambda =(\lambda_1,\ldots,\lambda_n)\in \mathbb{R}^n:\lambda_i>0, \; 1 \leq i \leq n,\}
$$  
and let $S$ be the group of permutations of the set $\{1,2,\ldots,n\}$.

\subsection{Interval exchanges}

Here our main reference is Veech [V1]. An exchange of $n$ intervals is a map which permutes $n$ given intervals. It is defined by a couple of parameters $(\lambda,\pi)\in \cC\times S$ in the following way.  Let $I^{\la}=[0,\Vert \lambda \Vert_1 )$, where $\Vert \lambda \Vert_1= \lambda_1+ \ldots + \lambda_n$. We set $\alpha_0(\lambda)=0$ and $\alpha_i(\lambda)=\lambda_1+\ldots  + \lambda_i$, for $1\leq i \leq n$. The points $\alpha_i(\lambda)$ partition the interval $I^{\la}$ into $n$ subintervals $I^{\la}_i=[\alpha_{i-1}(\la),\alpha_i(\la) )$ of length $\la_i$. Finally we use $\pi$ to permute those subintervals. We set $\lambda^{\pi}=(\lambda^{\pi}_1, \ldots , \lambda^{\pi}_n)$, where  $\lambda^{\pi}_i = \lambda_{\pi^{-1}(i)}$, for $1\leq i \leq n$. 

The $(\la,\pi)$-{\it interval exchange} is the one-one onto map $T=T_{(\lambda,\pi)}:I^{\la}\to I^{\la}$, defined by 
$$
T(x)=x- \alpha_{i-1}(\lambda) + \alpha_{\pi(i)-1}(\lambda^{\pi}), \; \mbox{ for } \; x \in I^{\lambda}_i, \; \mbox{ for } \; 1\leq i \leq n.
$$
The map $T$ acts as a translation on each subinterval $I^{\lambda}_i$ and thus $T$ preserves Lebesgue measure. Moreover we have $T(I^{\lambda}_i)= I^{\lambda^{\pi}}_{\pi(i)}$.

We say  that a permutation $\pi \in S$ is $ irreducible$, if $1\leq k\leq n$ and $\pi\{1,\ldots,k\}= \{1,\ldots,k\}$ imply $k=n$. In other words, for an irreducible permutation $\pi$,  if $x>0$ and $T_{(\lambda,\pi)}[0,x)=[0,x)$, then $x=\alpha_n(\lambda)$. We denote by $S_0$ the set of irreducible permutations of $S$.

If $\pi\in S$ is not irreducible, for every $\la\in\cC$ the corresponding $(\la,\pi)$-interval exchange may be seen as two independent exchanges of $k$ and $n-k$ intervals. In particular it is not ergodic with respect to Lebesgue measure. In what follows only irreducible permutations will be considered.

\subsection{Rauzy inductive process}

Here we follow the approach given in \cite[Section 2]{nru}. A vector $\la\in\cC$ is called {\it irrational} if its coordinates $\la_1,\ldots,\la_n$ are rationally independent. Let $T_{(\la,\pi)}$ be an interval exchange given by an irrational vector $\la$ and an irreducible permutation $\pi$. The so-called {\it Rauzy induction} assigns to $T_{(\la,\pi)}$ a first return map induced on a suitable subinterval of $I^{\la}$. We split $\mathcal{C}$ into two subcones   
$$
\mathcal{C}'=\{\lambda : \lambda_n\geq \lambda^{\pi}_n\} \; \mbox{ and } \;\mathcal{C}''=\{\lambda : \lambda^{\pi}_n \geq \lambda_n \}
$$
and define the induction for each of them separately.

If $\lambda \in \cC'$, we define
$$
T':[0, \alpha_{n-1}(\lambda^{\pi})) \rightarrow [0, \alpha_{n-1}(\lambda^{\pi}))
$$
to be the first return map induced by $T_{(\lambda,\pi)}$ on the interval $[0, \alpha_{n-1}(\lambda^{\pi}))$. A computation shows that $T'$ is still an $n$-interval exchange. The couple of parameters $(\la',\pi')\in\cC\times S$ corresponding to $T'$ is described as follows. Consider the $n\times n$-matrix
\be 
A_{\pi}'=\left( \begin{array}{ccccccc}
1 &   & & & & & \\
   & 1& & & & 0 & \\
   &    & & \ddots & && \\
   & 0 & &&&& \\
   &&&-1&&&1
\end{array} \right),
\ee
where $(A_{\pi}')_{n, \pi^{-1}n}=-1$. Then $\lambda'=A_{\pi}'\lambda$ and the permutation $\pi'$ is given by
\be \label{pi1}
\pi'(j)= 
\left\{ \begin{array}{ll}         
         \pi(j), & \mbox{if $\pi(j)\leq  \pi(n)$,} \\
         \pi(j)+1, & \mbox{if $\pi(n)<\pi(j)< n$,}\\
         \pi(n)+1, & \mbox{if $\pi(j)=n$.}
         \end{array}
\right.
\ee

If $\lambda \in \cC''$, we define 
$$
T'':[0, \alpha_{n-1}(\lambda)) \rightarrow [0, \alpha_{n-1}(\lambda))
$$
by inducing $T_{(\lambda,\pi)}$ on the interval $[0, \alpha_{n-1}(\lambda))$. Then $T''$ is also an $n$-interval exchange. We consider the $n\times n$-matrix
\be
A_{\pi}''=\left( \begin{array}{ccccccc}
1 &             &    & & & & \\
   & \ddots &    &  & && \\
   &              & 1&   &            & &   -1\\
   &              &0 & 0&\ldots  &0& 1\\
   &              &   & 1&            && \\
    &             &   &   & \ddots & & \\
    &             &   &   &             & 1& 0
\end{array} \right),
\ee
where $(A_{\pi}'')_{\pi^{-1}n,n}=-1$, and set $\lambda''=A_{\pi}''\lambda$. Let the permutation $\pi''$ be given by
\be \label{pi2}
\pi''(j)= 
\left\{ \begin{array}{ll}         
         \pi(j), & \mbox{if $j \leq  \pi^{-1}(n)$} \\
         \pi(n), & \mbox{if $j= \pi^{-1}(n)+1$}\\
         \pi(j-1), & \mbox{otherwise.}
         \end{array}
\right.
\ee
We have $T''=T_{(\lambda'',\pi'')}$.

The following lemma lets us iterate the inductive process described above.
 
\begin{lemma}[\cite{nru}, Lemma 2.4] Let $\la$ be irrational and $\pi$ irreducible. Then both $\la',\la''$ are irrational and both $\pi',\pi''$ irreducible. \end{lemma}

Let $\pi_0\in S_0$ be a fixed permutation and define $\mathcal{R}(\pi_0)$ to be the set of all permutations $\pi\in S_0$ which can be reached by the successive iterations of the Rauzy induction starting at some $T_{(\lambda,\pi_0)}$, $\la\in\cC$.  The set $\mathcal{R}(\pi_0)$ is called the {\it Rauzy class of permutations of $\pi_0$}, or the {\it Rauzy class of} $\pi_0$ for short.  

In order to study the possible sequences of permutations arising from this process,  we construct a directed graph $\mathcal{G}(\pi_0)$ whose nodes are the permutations $\pi\in \mathcal{R}(\pi_0)$. For every $\pi\in\cR(\pi_0)$ an arrow goes from $\pi$ to each of $\pi'$ and $\pi''$ given by (\ref{pi1}) and (\ref{pi2}) respectively. For $n=2$ we have only one Rauzy class whose graph consists of one node with two loops attached. The following lemma concerns the structure of the Rauzy graph for $n\geq 3$.

\begin{lemma}[\cite{nru}, Lemma 2.2 and 2.4] Let $\pi_0 \in S_0$. For every $\pi_1, \pi_2 \in \mathcal{G}(\pi_0)$ there is a path in $\mathcal{G}(\pi_0)$ starting at $\pi_1$ and reaching $\pi_2$. Moreover, every $\pi\in\cG(\pi_0)$ has exactly two followers and two predecessors in $\cG(\pi_0)$. \end{lemma}

\subsection{Rauzy induction algorithm}

Let $ \mathcal{R}$ be a Rauzy class in $S_0$. The inductive process described in the previous subsection defines an algorithm $\cI$ acting on the parameter space $\cC\times \cR$ by 
\be \label{I}
\mathcal{I}:  (\lambda,\pi) \in \mathcal{C} \times \mathcal{R}  \longmapsto \left\{ \begin{array}{lll} (\lambda',\pi')  & \text{if} & \lambda_n > \la_n^{\pi}, \\ (\la'',\pi'') & \text{if} & \la_n <\la_n^{\pi}. \end{array} \right.
\ee 
It is called the {\it Rauzy induction} of interval exchange transformations. 

The space $\mathcal{C} \times \mathcal{R}$ is endowed with Lebesgue measure denoted by $\mu$.

\begin{thm}[\cite{veech2}, Theorem 1.6] 
For every Rauzy class $\mathcal{R}$, the map $\mathcal{I}$ is ergodic on $\mathcal{C} \times \mathcal{R}$ with respect to Lebesgue measure.\end{thm}

In order to illustrate the definition of $\cI$, we will now describe explicitly its action in the easiest cases $n=2,3$. In what follows, we represent the permutations in the form $\pi=(\pi^{-1}(1),\ldots,\pi^{-1}(n))$.
\\

\noindent {\bf 1.} In the case of $n=2$, we have only one irreducible permutation on two letters, the transposition $(2,1)$, which constitutes its own Rauzy class. The action of $\cI$ on the second coordinate is thus trivial. On the first coordinate $\cI$ acts as the Euclidean algorithm $\cE$ defined by (\ref{E1}):
$$ \cI (\la, (2,1)) = \left\{ \begin{array}{ll}         
         
         ((\lambda_1-\lambda_2,\lambda_2), (2,1)), & \mbox{if $\lambda_1 >\lambda_2$}, \\
         ((\lambda_1,\lambda_2-\lambda_1), (2,1)), & \mbox{if $\lambda_2 > \lambda_1$}.
\end{array}
\right. $$
We have the corresponding Rauzy graph:

\vspace{0.5cm} 

\begin{centering}
\begin{tikzpicture}[->,>=stealth',shorten >=1pt,auto,node distance=2.8cm,
                    semithick]

  \node[state] (A)                    {$21$};

\path  (A) edge [loop right]        (A);

     \path    (A) edge [loop left]  (A);

\end{tikzpicture}

\end{centering} 

\vspace{0.5cm}

\noindent {\bf 2.} For $n=3$ we also have an unique Rauzy class that contains all irreducible permutations on three letters, namely $(231)$, $(321)$, $(312)$. The Rauzy induction is decribed as follows:
$$
\mathcal{I}(\lambda,(2,3,1))=
\left\{ \begin{array}{ll}         
         ((\lambda_1,\lambda_2,\lambda_3-\lambda_1), (2,3,1)), & \mbox{if $\lambda_3 > \lambda_1$}, \\
         ((\lambda_1-\lambda_3,\lambda_3,\lambda_2), (3,2,1)), & \mbox{if $\lambda_1 >\lambda_3$},
\end{array}
\right.
$$ 
$$
\mathcal{I}(\lambda,(3,2,1))=
\left\{ \begin{array}{ll}         
         ((\lambda_1,\lambda_2,\lambda_3-\lambda_1), (3,1,2)), & \mbox{if $\lambda_3 > \lambda_1$}, \\
         ((\lambda_1-\lambda_3,\lambda_3,\lambda_2), (2,3,1)), & \mbox{if $\lambda_1>\lambda_3$},
\end{array}
\right.
$$
$$
\mathcal{I}(\lambda,(3,1,2))=
\left\{ \begin{array}{ll}         
         ((\lambda_1,\lambda_2,\lambda_3-\lambda_2), (3,2,1)), & \mbox{if $\lambda_3 > \lambda_2$}, \\
         ((\lambda_1,\lambda_2-\lambda_3,\lambda_3), (3,1,2)), & \mbox{if $\lambda_2>\lambda_3$}.
\end{array}
\right.
$$

From this expression the Rauzy graph can be deduced.

\vspace{0.5cm}

\begin{centering}

\begin{tikzpicture}[->,>=stealth',shorten >=1pt,auto,node distance=2.8cm,
                    semithick]

  \node[state] (A)                 {$231$};
 \node[state] (B)    [right of=A]                {$321$};
  \node[state] (C)    [right of=B]                {$312$};

\path  (C) edge [loop right]        (C);

     \path    (A) edge [loop left]  (A);
                
    \path (A) edge [out=30, in=150]  (B);
    
     \path (B) edge [in=-30, out=-150] (A);
     
      \path (B) edge [out=30, in=150] (C);
      
       \path (C) edge [in=-30, out=-150] (B);

\end{tikzpicture}

\end{centering}

\vspace{0.5cm}

\noindent {\bf 3.} One may check that for $n=4$ we get two distinct Rauzy classes of irreducible permutations, one generated by $(4321)$ and the other by $(3412)$.

\section{Rauzy classes}

We will need more information about permutations within a given Rauzy class. A permutation  $\pi \in S_0$ is said to be $standard$, if $\pi(1)=n$ and $\pi(n)=1$. 

\begin{lemma}[\cite{rauzy}] \label{sevenone} Every Rauzy class contains a standard permutation. \end{lemma}

The notion of standard permutation was rediscovered in \cite{nru}, where it was noticed that  the existence of standard permutations  in every Rauzy class was a central tool to prove the weak-mixing property of interval exchanges (see also \cite{avila-forni}). Here standard permutations will also be used. First we prove a  technical lemma which concerns permutations which are fixed by Rauzy induction. 

\begin{lemma} \label{loop} Let $\pi \in S_0$ be such that $\pi(n-1)=n$ and $\pi''$ be the permutation defined by (\ref{pi2}). Then $\pi''=\pi$. \end{lemma}

\begin{proof} We have $\pi^{-1}(n)=n-1$. By (\ref{pi2}), $\pi''$ reduces to 
$$
\pi''(j)= 
\left\{ \begin{array}{ll}         
         \pi(j), & \mbox{if $j \leq  n-1$}, \\
         1, & \mbox{if $j= n-1+1=n$},
                 \end{array}
\right.
$$ 
which implies $\pi''=\pi$. \end{proof}

The above lemma proves that, at such node, the Rauzy graph has a loop. We call {\it loop permutation} an irreducible permutation $\pi$ with $\pi(n-1)=n$.

\begin{lemma} \label{7.3} Every Rauzy class contains a  loop permutation.\end{lemma}

\begin{proof}  Let $\cR$ be a Rauzy class. By  Lemma \ref{sevenone}, $\cR$ contains  a standard permutation $\sigma$. We will show that there is a loop permutation $\pi$  in the orbit of $\sigma$ under $\cI$.  The idea of the proof is depicted in the figures in \cite[p.1192]{nru} and corresponds to the case $i=n-1$ therein.

Let $\lambda\in \mathcal{C}$ satisfy $\lambda_n> \lambda_1+\ldots + \lambda_{n-1}$. For $k\geq 1$ set $(\la^{(k)},\sigma^{(k)})=\cI^k(\la,\sigma)$. Since $\sigma$ is standard, we have $\sigma(n)=1<\sigma(n-1)<n=\sigma(1)$. For $1\leq j \leq  n-\sigma(n-1)$ one gets
$$
\lambda^{(j)}=(\lambda_1,\ldots , \lambda_{n-1},\lambda_n-\lambda_{\sigma^{-1}(n)}-\ldots - \lambda_{\sigma^{-1}(n+1-j)}),
$$
$\sigma^{(j)}(n)=1$ and $\sigma^{(j)}(n-1)=\sigma(n-1)+j$. This implies that $\pi = \sigma^{(n-\sigma(n-1))}$ is a loop permutation. Moreover, since $\pi$ is obtained as an image of $\sigma$ by the Rauzy induction $\cI$, it belongs to the same Rauzy class $\cR$. \end{proof}

\section{Cone partitions} 

Fix a permutation $\pi_0\in S_0$. For every irrational $\la\in\cC$, one may consider the iterations $\cI^k(\la,\pi_0)=(\la^{(k)},\pi^{k}_{\la})$ of the Rauzy induction algorithm (\ref{I}). This generates an infinite sequence of permutations 
\be \label{piseq} (\pi_{\la})= \pi_0,\pi^{1}_{\la},\ldots,\pi^{k}_{\la},\ldots \ee
which is an infinite path in the Rauzy graph $\cG(\pi_0)$. Together with $(\pi_{\la})$ we get an infinite sequence of matrices 
$$ 
A^{1}_{\la}, A^{2}_{\la}, \ldots,A^{k}_{\la},\ldots,
$$
such that $\la^{(k)}= A^{k}_{\la}\cdots A^{1}_{\la} \la$, for $k\geq 1$. We set
$$
A_{\la}^{(k)}=A^{k}_{\la} \cdots A^{2}_{\la} A^{1}_{\la} \; \text{and} \; B^{(k)}_{\la}=(A_{\la}^{(k)})^{-1},$$ where $B^{(k)}_{\la}$ is a non-negative matrix.

Conversely, to any infinite path $\pi_0, \pi^1,\pi^2,\ldots$ in $\cG(\pi_0)$, there corresponds a nonempty closed subset of vectors $\la\in\cC$ that generate that path {\it via} Rauzy induction. To be more precise, let $\pi_0,\pi^1, \ldots, \pi^k$ be a finite path in $\mathcal{G}(\pi_0)$ and define 
$$C_{\pi_0}^{\pi^1,\ldots, \pi^k} =\{\la\in\cC :  \pi^{i}_{\la}=\pi^i, 1\leq i\leq k\}.$$
Every $\la\in {C}_{\pi_0}^{\pi^1,\ldots, \pi^k}$ generates the same beginning of the sequence of matrices $A^1_{\la},A^2_{\la},\ldots,A^k_{\la}$. The set ${C}_{\pi_0}^{\pi^1, \ldots, \pi^k}$ is an Euclidean cone satisfying
$${C}_{\pi_0}^{\pi^1, \ldots, \pi^k}=B^{(k)}_{\la}(\cC)= \{\alpha_1\l_1^{(k)}(\la) + \ldots +\alpha_n \l_n^{(k)}(\la):\alpha_i\geq 0,1\leq i \leq n\},
$$ where  $\l_1^{(k)}(\la), \ldots , \l_n^{(k)}(\la)$ stand for the column-vectors of the matrix $B^{(k)}_{\la}$. 

Since every permutation in $\cG(\pi_0)$ has exactly two followers, the finite path  $\pi_0,\pi^1, \ldots, \pi^k$ may be continued in two different ways, choosing $\pi^{k+1}=(\pi^k)'$ or $\pi^{k+1}=(\pi^k)''$. This choice splits the cone ${C}_{\pi_0}^{\pi^1,\ldots, \pi^k}$ into two nontrivial subcones ${C}_{\pi_0}^{\pi^1, \ldots,\pi^k,  (\pi^k)'}$ and ${C}_{\pi_0}^{\pi^1, \ldots, \pi^k,(\pi^k)''}$.

For $k\geq 1$ fixed, let $\cP^{(k)}(\pi_0)$ be the family of cones ${C}_{\pi_0}^{\pi^1,\ldots, \pi^k}$, where $\pi^1,\ldots,\pi^k$ runs through all possible $k$-paths in $\cG(\pi_0)$ starting at $\pi_0$. The family $\cP^{(k)}(\pi_0)$ forms a partition of $\cC\times \{\pi_0\}$ into $2^k$ subcones and $\cP^{(k+1)}(\pi_0)$ is  a refinement of $\cP^{(k)}(\pi_0)$ for every $k\geq 1$. 
\\

Let $\lambda \in \mathcal{C}$ be irrational. For each $k \geq 1$, we denote by ${C}^{(k)}(\lambda,\pi_0)$ the unique subcone of the partition  $\mathcal{P}^{(k)}(\pi_0)$ which contains $(\lambda,\pi_0)$.  We need the following result.

\begin{lemma}[\cite{kerckhoff}, Corollary 1.9] \label{ker}  Let $\pi_0 \in S_0$. There is a positive constant $c = c(\pi_0)$ such that for almost every $\lambda \in \mathcal{C}$ there are infinitely many integers $k\geq 1$ with
\begin{enumerate}
\item $\pi^k_{\la}=\pi_0$,\\
\item $\displaystyle \frac{max_i \Vert \l^{(k)}_i (\lambda)\Vert_1}{min_i \Vert \l^{(k)}_i (\lambda)\Vert_1} \leq c(\pi)$, where $\l^{(k)}_1(\la) ,\ldots, \l^{(k)}_n(\la)$ are the column-vectors  of the matrix $B^{(k)}_{\la}$ which generate the subcone ${C}^{(k)}(\lambda,\pi_0)$. \end{enumerate} \end{lemma}

\begin{cor} \label{cor-ker} Let $\pi_0 \in S_n^0$. There exists a partition of $\mathcal{C}\times \{\pi_0\}$ whose elements are subcones $\mathcal{C}^{(k)}(\lambda,\pi_0)$ which satisfy Lemma \ref{ker}. \end{cor} 

\begin{proof} The argument is easily adapted from the one used in the proof of Corollary \ref{distortionthm}. \end{proof}

The next lemma is equivalent to the unique ergodicity of almost every interval exchange.

\begin{lemma}[\cite{kerckhoff}, Theorem 1.10, \cite{masur}, \cite{veech1}] \label{3.5} Let $\pi_0\in S_0$. For almost every $\lambda \in \mathcal{C}$, 
$$
\cap_{k\geq 1} {C}^{(k)}(\lambda,\pi_0)= \{\alpha \lambda: \alpha \geq 0\}.
$$ \end{lemma}

For $\pi\in S_0$ let $\cP(\pi)$ be a partition of $\cC\times \{\pi\}$ given by Corollary \ref{cor-ker}. The following lemma is a generalization of Corollary \ref{distortionthm} to the case of Rauzy induction.

\begin{lemma} \label{part}  Let $\pi \in S_0$ be a loop permutation. For every $N\geq 1$ there exists a partition $\cP_N$ of $\mathcal{C}\times \{\pi\}$ which is a refinement of the partition $\mathcal{P}(\pi)$ and satisfies the following properties:
\begin{enumerate}
\item its elements are subcones of type $C^{(k)}(\la,\pi)$,\\
\item $\displaystyle \frac{ \Vert \l_n^{(k)} (\lambda)\Vert_1}{ \Vert \l_{n-1} ^{(k)}(\lambda)\Vert_1} > N$, where  $\l_1^{(k)}(\la) ,\ldots, \l_n^{(k)}(\la)$ are the column-vectors  of the matrix $B^{(k)}_{\la}$ which generate the subcone ${C}^{(k)}(\la,\pi)$. \end{enumerate} \end{lemma}

\begin{proof} Let  ${C}' \times \{\pi\}={C}^{(k)}(\la,\pi)\times \{\pi\}$ be an element of the partition $\mathcal{P}(\pi)$ and $\l_1 ,\ldots, \l_n$ be the column-vectors  of the matrix $B=B^{(k)}_{\la}$ which defines the subcone ${C}'=B(\mathcal{C})$. 

We recall that $\pi^k_{\la}=\pi$ (see Lemma \ref{ker}).  Since $\pi$ is a loop permutation, we may continue the path $\pi,\pi^1_{\la},\ldots,\pi^k_{\la}$ choosing $\pi^j=\pi$ for $k+1\leq j\leq k+N$. Let ${C}' _N$ be the subcone of ${C}' $ corresponding to this path. It is generated by a matrix $B_N$ whose column-vectors are  $\l_1 ,\ldots, \l_{n-1}, \l_n+ N \l_{n-1}$. We have 
$$\frac{\Vert l_n+N l_{n-1} \Vert_1}{\Vert l_{n-1} \Vert_1} = N + \frac{\Vert l_n \Vert_1}{\Vert l_{n-1} \Vert_1}\geq N +\frac{1}{c(\pi)}.$$

To show that almost every $\la\in C'$ belongs to such a cone, we will show that the cone $C'_N$ occupies a large proportion of the volume of the cone $C'$. To this end, let 
\be \label{simplex} 
\Delta_{n-1}=\{x \in \mathcal{C}:x_1+\ldots + x_n=1\}
 \ee
 be the $(n-1)$-dimensional simplex. We project radially the subcones  ${C}'$ and ${C}' _N$ on $\Delta_{n-1}$. By \cite[Lemma 3.2]{mns}, the ratio of the volumes of the projections of $C'$ and $C'_N$ equals
$$
\frac{ \Vert \l_n \Vert_1 \cdots \Vert \l_n \Vert_1}{  \Vert \l_1 \Vert_1 \cdots \Vert \l_{n-1} \Vert_1 \Vert \l_n+ K\l_{n-1}  \Vert_1}.
$$ 
It is bounded from below by  $\frac{1}{1+Nc(\pi)}$, which is a constant independent of the initial cone $C'$. We may thus deduce that iterating the same construction on $C'\setminus  C'_N$ will result in a desired partition of $C'$. \end{proof}

\section{Proof of Theorem \ref{rauzy}} 

We will adapt the proof of Theorem \ref{euclid} to the multidimensional case of Rauzy induction. 

Let $\mathcal{R}$ be a Rauzy class and $\pi \in \mathcal{R}$ be a loop permutation. Let $\Omega \subset \mathcal{C} \times \mathcal{R}$ be a set of positive measure which, without loss of generality, is assumed to be contained in  $ \mathcal{C} \times \{\pi \}$. As the map $\mathcal{I}$ is non-singular and ergodic, in order to prove that it is exact, according to Lemma \ref{exact-lem},  it suffices to show that there exists $k\geq 1$ such that 
\be \label{7.1} 
\mu(\mathcal{I}^{k+1}(\Omega) \cap \mathcal{I}^{k}(\Omega))>0. 
\ee

Let $\lambda^0\in \Omega$ be a Lebesgue density point of $\Omega$ satisfying Lemma \ref{3.5}. Let $\rho>0 $ be such that the ball of center  $\lambda^0$ and radius $\rho$ is entirely contained in $\mathcal{C}$. Let $D_{\rho}=\{\lambda^0+x : x\in\cC, x_1\la^0_1+\ldots + x_n\la^0_n=0, \Vert x \Vert_1 \leq \rho \}$ which is an $(n-1)$-dimensional ball centered at $\la^0$ of radius $\rho$. Next we define  a section of a cylindrical cone 
$$
\Sigma(\lambda^0,\rho)=\{t\lambda: \lambda \in D_{\rho},  \delta \leq t \leq 1 \}, \; \mbox{ where } \delta= 1-\frac{\rho}{\Vert \lambda^0 \Vert_1 } .
$$
Given $\varepsilon>0$, for $\rho>0$ sufficiently small,
$$
\mu(\Omega \cap  \Sigma(\lambda^0,\rho)) > (1-\varepsilon) \mu(\Sigma(\lambda^0,\rho)),
$$ by Lebesgue density theorem.

Let $N\geq 1$ and $\cP_N(\pi)$ be a partition of $\cC\times \{\pi\}$ given by Lemma  \ref{part}. The sets 
$$\Sigma(\la^0,\rho)\cap C_N, \ \ C_N\in \cP_N(\pi),$$
partition $\Sigma(\la^0,\rho)$ into a family of polyhedral slices. In virtue of Lemma  \ref{3.5}, taking a refinement of $\cP_N(\pi)$ if necessary, one may assume that there exists a subcone $C_N\in \cP_N(\pi)$ such that 
$$\mu(\Omega\cap \Sigma(\la^0,\rho)\cap C_N)\geq(1-2\varepsilon)\mu(  \Sigma(\la^0,\rho)\cap C_N)$$
(see the proof of Lemma \ref{Q}). 

Recall that $C_N=C^{(k)}(\la,\pi)$ for some $k\geq 1$ and $\la\in\cC$. In particular, this implies 
$$\cI^k(C_N\times\{\pi\})=\cC\times \{\pi\}.$$
Let $l_1,\ldots,l_n$ be the column-vectors of the matrix $B^{(k)}_{\la}$ generating the polyhedral cone $C_N$. As in the case of the Euclidean algorithm, the slice $\Sigma(\la^0,\rho)\cap C_N$ may be described as
$$\Sigma(\la^0,\rho)\cap C_N = \{x_1 t_1 l_1 + \ldots + x_n t_n l_n : x_i\geq 0, x_1+\ldots+x_n=1, \alpha_i\leq t_i\leq \beta_i\},$$
where 
\be \label{ratios2}  \frac{\alpha_i}{\beta_i}=\delta \ \text{for} \ 1\leq i\leq n \ \text{and} \ \frac{\alpha_{n-1}}{\alpha_n}=\frac{\beta_{n-1}}{\beta_n}\geq \frac{N}{2}, \ee
if $\rho$ is small enough (see Lemma  \ref{ratios}). 

The matrix $A^{(k)}_{\la}$ associated to $C_N$ sends the vectors $l_1,\ldots,l_n$ onto the canonical basis of $\R^n$. Let $P$ stand for the polyhedral slice $A^{(k)}_{\la}( \Sigma(\la^0,\rho)\cap C_N)$. We have 
$$\cI^k((\Sigma(\la^0,\rho)\cap C_N)\times \{\pi\})=P\times \{\pi\}$$
and 
$$ P=\{ (x_1 t_1,\ldots, x_n t_n)\in\cC : x_i\geq 0, x_1+\ldots+x_n=1, \alpha_i\leq t_i\leq \beta_i\}.$$
A calculation shows that 
$$\mu(P)=\frac{1}{n !} \beta_1\beta_2\cdots \beta_n - \frac{1}{n !} \alpha_1\alpha_2\cdots \alpha_n=\frac{1}{n !}(1-\delta^n)\beta_1\beta_2\cdots \beta_n.$$

As in the case of the Euclidean algorithm, we are interested in the subset $P_+$ of $P$ defined by $P_+=\{\la\in P: \la_{n-1}>\la_n\}$. In virtue of (\ref{ratios2}), a calculation analogous to (\ref{?}) gives
$$\mu(P_+)\geq \frac{N}{N+2}\mu(P).$$ 
For $N$ large enough we may thus assume
$$\mu(\Omega\cap P_+)\geq (1-3\varepsilon)\mu(P_+).$$

We want to show that $\cI(P_+\times \{\pi\})$ intersects $P_+\times \{\pi\}$ and that the volume of this intersection is large enough to imply (\ref{7.1}). First, since $P_+$ is contained in the set $\{\la\in \cC: \la_{n-1}>\la_n\}$ and $\pi$ is a loop permutation, we have $\cI(P_+\times \{\pi\})\subset \cC\times \{\pi\}$. It is thus enough to show the intersection property on the first coordinate. To this end, we remark that the only coordinate that changes when applying $\cI$ on $P_+$ is $\la_{n-1}$. Moreover, the action on the couple of coordinates $(\la_{n-1},\la_n)$  corresponds to that of the Euclidean algorithm $\cE$. The argument of Lemma \ref{intersectionlem} is then valid also in this case. We get
$$\frac{\mu(P_+\times\{\pi\} \cap \cI( P_+\times\{\pi\}))}{ \mu(P_+\times\{\pi\})} \to 1$$
as $N\to\infty$. Choosing $N$ large enough we get the intersection property (\ref{7.1}). The Rauzy induction is exact with respect to Lebesgue measure. Theorem \ref{rauzy} is proved. \hfill $\Box$

\section{Remarks}

In many cases, in particular in \cite{veech1}, instead of the homogenous algorithm $\mathcal{I}$ defined by (\ref{I}), a normalized version  is considered, for example its radial projection on the simplex $\Delta_{n-1}$  (\ref{simplex}),
$$
\tilde{\mathcal{I}}(\lambda,\pi) \in \Delta_{n-1} \times \mathcal{R}  \longmapsto (\frac{\lambda'}{\Vert  \lambda' \Vert_1},\pi') \in\Delta_{n-1} \times \mathcal{R}.
$$
The map $\tilde{\mathcal{I}}$ is conservative and ergodic with respect to Lebesgue measure on the simplex $\Delta_{n-1}$. The following result may be deduced from Theorem \ref{rauzy}.

\begin{cor} The map $\tilde{\mathcal{I}}$ is exact with respect to Lebesgue measure. \end{cor}

For completeness, we give examples of multidimensional continued fraction algorithms which are adapted to our approach and should satisfy the intersection property.

First we define the map
$$
\sigma: \la \in \cC \mapsto \sigma( \la)=(\la_{\sigma_{\la}(1)}, \dots , \la_{\sigma_{\la}(n)})  \in \cC,
$$
where $\sigma_{\la}$ arranges  the coordinates $\la_1,\ldots , \la_n$ in non decreasing order. We recall that, if $\la$ is irrational,  the permutation $\sigma_{\la}$  is unique. 

Let $1\leq i \leq n-1$ and  $\mathcal{T}_i$ be the homogeneous algorithm given by
$$
\la \in \cC \mapsto \sigma_{\la}^{-1}(\la_{\sigma_{\la}(1)}, \dots , \la_{\sigma_{\la}(n-1)}, \la_{\sigma_{\la}(n)}-\la_{\sigma_{\la}(i)})  \in \cC.
$$
The map $\mathcal{T}_i$ is nonsingular and dissipative. The subcones $C_j=\{\la \in \cC: \sigma_{\la}(j)=i\}$, for $1\leq j \leq n$, define a partition of the cone $\cC$. The map $\mathcal{T}_i$ satisfies a Markov partition property: $\mathcal{T}_i(C_j)=\cC$, for $j=1,\ldots , n$.

The map  $\mathcal{T}_{n-1}$ is known as {\it homogeneous Brun algorithm} (see \cite[p.45]{schweiger}) and the map $\mathcal{T}_{1}$ is the called  {\it homogeneous Selmer algorithm} (see \cite[p.45]{schweiger}). Our definitions coincide with those of the reference up to a permutation, however, as far as  ergodicity and exactness are concerned, they bare the same properties.

Another example of multidimensional algorithm that could be studied this way is the {\it Jacobi-Perron algorithm}.  Following  \cite[p.24]{schweiger},  it is convenient to define it on a subcone of $C$. Let $\tilde{\cC}=\{ \la \in \cC: \la_2>0, \la_1\geq \la_i, \forall \ 1\leq i \leq n \}$ and $[x]$ be the integer part of $x \in \R$.   The Jacobi-Perron algorithm, denoted by $\mathcal{J}$, is the map defined by 
$$
\mathcal{J}: \la \in \tilde{\cC} \mapsto  \mathcal{J}( \la)=(\la_2, \la_3-a_2\la_2, \ldots , \la_n-a_{n-1}\la_2, \la_1-a_{n}\la_2) \in \cC,
$$
where $a_j=[\la_{j+1}/\la_2]$, for $2\leq j \leq n-1$ and  $a_n=[\la_{1}/\la_2]$. 

The Jacobi-Perron algorithm is ergodic and has a finite invariant measure absolutely continuous with respect to Lebesgue measure. Although it is not homogeneous, it may be seen as a suitable first return time of a homogeneous algorithm defined in \cite[Section 3.1]{arnaldo-bb}. This underlying algorithm is adapted to our approach and we conjecture that $\mathcal{J}$ is exact.

\bibliographystyle{plain}

\bibliography{exact}

\vskip1cm
\begin{flushleft}
Tomasz Miernowski \;\; and \;\; Arnaldo Nogueira\\
Institut de Math\'ematiques de Luminy\\
163, avenue de Luminy, Case 907\\
13288 Marseille Cedex 9, France\\

\smallskip
E-mail: {\tt miernow@iml.univ-mrs.fr\;\;} and {\;\;\tt nogueira@iml.univ-mrs.fr}

\end{flushleft}

\end{document}